\documentclass{amsart}
\usepackage{amssymb,color,comment}
\usepackage[shortlabels]{enumitem}
\usepackage[ascii]{inputenc}
\usepackage[all]{xy}
\usepackage{hyperref}

\title[Controlling characteristics while collapsing cardinals]{Controlling classical cardinal characteristics while collapsing cardinals}

\author[M. Goldstern]{Martin Goldstern}
\address{Institut f\"ur Diskrete Mathematik und Geometrie, Technische Universit\"at Wien, Wiedner Hauptstrasse 8-10/104, 1040 Vienna, Austria.}
\email{goldstern@tuwien.ac.at}
\urladdr{http://www.tuwien.ac.at/goldstern/}

\author[J. Kellner]{Jakob Kellner}
\address{Institut f\"ur Diskrete Mathematik und Geometrie, Technische Universit\"at Wien, Wiedner Hauptstrasse 8-10/104, 1040 Vienna, Austria.}
\email{kellner@fsmat.at}
\urladdr{http://dmg.tuwien.ac.at/kellner/}

\author[D. Mej\'{i}a]{Diego A. Mej\'{i}a}
\address{Creative Science Course (Mathematics), Faculty of Science, Shizuoka University, Ohya 836, Suruga-ku, Shizuoka-shi, Japan 422-8529.}
\email{diego.mejia@shizuoka.ac.jp}
\urladdr{http://www.researchgate.com/profile/Diego\_Mejia2}

\author[S. Shelah]{Saharon Shelah}
\address{Einstein Institute of Mathematics, Edmond J. Safra Campus, Givat Ram, The Hebrew University of Jerusalem, Jerusalem, 91904, Israel, and Department of Mathematics, Rutgers University, New Brunswick, NJ 08854, USA.}
\email{shlhetal@math.huji.ac.il}
\urladdr{http://shelah.logic.at}

\thanks{This work was supported by the following grants:
Austrian Science Fund (FWF): project number I3081
(first author) and P30666
(second author);
Grant-in-Aid for Early Career Scientists 18K13448, Japan Society for the Promotion of Science (third author); Israel Science Foundation (ISF) grant no: 1838/19 (fourth author). This is publication number E87 of the fourth author.}


\newcommand{\LCU}{\textnormal{LCU}}
\newcommand{\COB}{\textnormal{COB}}
\DeclareMathOperator{\add}{add}
\DeclareMathOperator{\cov}{cov}
\DeclareMathOperator{\non}{non}
\DeclareMathOperator{\cof}{cof}
\DeclareMathOperator{\cf}{cof}
\DeclareMathOperator{\crit}{cr}

\DeclareMathOperator{\col}{Coll}

\DeclareMathOperator{\dom}{dom}
\newcommand{\Null}{\mathcal N}
\newcommand{\Meager}{\mathcal M}
\newcommand{\addN}{{\ensuremath{\add(\Null)}}}
\newcommand{\cofN}{{\ensuremath{\cof(\Null)}}}
\newcommand{\covN}{{\ensuremath{\cov(\Null)}}}
\newcommand{\nonN}{{\ensuremath{\non(\Null)}}}
\newcommand{\addM}{{\ensuremath{\add(\Meager)}}}
\newcommand{\cofM}{{\ensuremath{\cof(\Meager)}}}
\newcommand{\covM}{{\ensuremath{\cov(\Meager)}}}
\newcommand{\nonM}{{\ensuremath{\non(\Meager)}}}

\newcommand{\cfrak}{\mathfrak{c}}
\newcommand{\bfrak}{\mathfrak{b}}
\newcommand{\dfrak}{\mathfrak{d}}
\newcommand{\gfrak}{\mathfrak{g}}
\newcommand{\hfrak}{\mathfrak{h}}

\newcommand{\mfrak}{\mathfrak{m}}
\newcommand{\pfrak}{\mathfrak{p}}
\newcommand{\tfrak}{\mathfrak{t}}
\newcommand{\xfrak}{\mathfrak{x}}
\newcommand{\precal}{\textnormal{cal}}

\newcommand{\plike}{$\mathfrak{t}$-like}
\newcommand{\tlike}{\plike}
\newcommand{\mlike}{$\mathfrak{m}$-like}
\newcommand{\hlike}{$\mathfrak{h}$-like}

\newcommand{\vA}{\textnormal{\texttt{vA}}}
\newcommand{\vB}{\textnormal{\texttt{vB}}}

\newcommand{\la}{\langle}
\newcommand{\ra}{\rangle}
\newcommand{\Qhor}{Q^2}

\theoremstyle{plain}
  \newtheorem{theorem}[equation]{Theorem}
  \newtheorem{corollary}[equation]{Corollary}
  \newtheorem{lemma}[equation]{Lemma}
   
   \newtheorem{fact}[equation]{Fact}

\theoremstyle{definition}
  \newtheorem{definition}[equation]{Definition}
  \newtheorem{example}[equation]{Example}
  \newtheorem{remark}[equation]{Remark}
  \newtheorem{remarks}[equation]{Remark}
  \newtheorem{notation}[equation]{Notation}
  
  \newtheorem{assumption}[equation]{Assumption}

\numberwithin{equation}{section}

\begin{document}
\begin{abstract}
    Given a forcing notion $P$ that forces certain values to
    several classical cardinal characteristics of the reals,
    we show how we can compose $P$ with a collapse (of a cardinal $\lambda>\kappa$
    to $\kappa$) such that the composition still
    forces the previous values to these characteristics.

    We also show how to force distinct values to $\mathfrak m$,
    $\mathfrak p$ and $\mathfrak h$ and also  keeping 
    all the values in Cicho\'n's diagram distint,  using the Boolean Ultrapower method.
    (In our recent paper \emph{Controlling cardinal characteristics without adding reals}
    the same was done for the newer Cicho\'n's Maximum construction, which avoids large cardinals.)
\end{abstract}
\maketitle


\section*{Introduction}

Cicho\'n's diagram (see Figure~\ref{fig:cichon})
lists ten  cardinal characteristics of the continuum,
which we will call \emph{Cicho\'n-characteristics}
(where we ignore the two
``dependent'' characteristics $\addM=\min(\covM,\bfrak)$ and $\cofM=\max(\nonM,\dfrak)$).

In many constructions that force given values
to such characteristics we actually get something stronger,
which we call ``strong witnesses'' (the objects $\bar f$ and $\bar g$ in Definition~\ref{def:lcu.cob}).

In this paper, we show how to collapse cardinals 
while preserving the strongly witnessed values for 
Cicho\'n-characteristics 
(and certain other types of characteristics).

With \emph{Cicho\'n's Maximum} we denote the statement
``all
Cicho\'n-characteristics (including $\aleph_1$ and the continuum) are pairwise different''.
In~\cite{GKMS2} we show how to force Cicho\'n's Maximum
(without using large cardinals).

In~\cite{GKMS1} we investigate how to preserve and how to change classical
cardinal characteristics of the continuum in NNR extensions, i.e., extensions that do not add reals; and we show how this gives 13 pairwise different ones: ten from Cicho\'n's Maximum,
plus $\mfrak$, $\pfrak$ and $\hfrak$ (see Definition~\ref{DefChar1}). This construction is
based on~\cite{GKMS2} (and accordingly does not use large cardinals).

The original Cicho\'n's Maximum construction~\cite{GKS}
uses Boolean ultrapowers (which makes large cardinals necessary).
It turns out that it is possible to add
$\mfrak$, $\pfrak$ and $\hfrak$ to this construction as well (see Figure~\ref{fig:result}); and as this construction seems 
interesting in its own right, we give the details in this paper.
(But note that the result of~\cite{GKMS1} is stronger in the sense
that we do not require large cardinals there; 
the advantage of the result in this paper is that 
we can obtain singular values for $\covM$ and $\dfrak$
in certain circumstances.)

Annotated Contents:

We will briefly review the Boolean ultrapower constructions
in \textbf{Section~\ref{sec:prelim}}.
We also describe how we can start with alternative initial forcings
(for the left hand side of Cicho\'n's diagram), which for example allow us to get Cicho\'n's Maximum plus distinct values for $\mfrak$, $\pfrak$ and $\hfrak$, and allows two Cicho\'n-characteristics to be singular, namely, $\cfrak$ and either $\covM$ or $\dfrak$ (the latter when the Cicho\'n's Maximum construction is based on~\cite{diegoetal}). This contrasts the constructions in~\cite{GKMS1} without large cardinals, where only $\cfrak$ is allowed to be singular.

Part of the following Sections
are parallel to~\cite{GKMS1}, and we will
regularly refer to 
that paper; this applies
in particular to \textbf{Section~\ref{sec:nnr}} (and parts of Subsection~\ref{subsec:blass}),
where we describe some classes of cardinal characteristics,
and their behaviour under no-new-reals extension.

In \textbf{Section~\ref{sec:appl}} we show
how to add
$\mfrak$, $\pfrak$ and $\hfrak$ to the Boolean ultrapower construction.

Also, the Boolean ultrapower method produces large gaps between the Cicho\'n values of the left hand side:
the $\kappa_i$ in Figure~\ref{fig:1022} are strongly compact (in the ground model; so as cofinalities are preserved they are still weakly inaccessible in the extension).

This was the original motivation for the 
main result of this paper:
In \textbf{Section~\ref{sec:coll}} we show how we can 
collapse cardinals while keeping values for characteristics
that are either strongly witnessed or small.
(In particular, we can get rid of the gaps 
that necessarily arise in the Boolean ultrapower construction.)


\newcommand{\mye}{*+[F.]{\phantom{\lambda}}}
\begin{figure}
  \centering
\[
\xymatrix@=4.5ex{
&            \covN\ar[r] & \nonM \ar[r]      &  \mye \ar[r]     & \cofN\ar[r] &2^{\aleph_0} \\
&                               & \mathfrak b\ar[r]\ar[u]  &  \mathfrak d\ar[u] &              \\
  \aleph_1\ar[r] & \addN\ar[r]\ar[uu] & \mye\ar[r]\ar[u] &  \covM\ar[r]\ar[u]& \nonN\ar[uu]
}
\]
    \caption{\label{fig:cichon}Cicho\'n's diagram with the
two ``dependent'' values removed, which are
$\addM=\min(\mathfrak b, \covM)$
and $\cofM=\max(\nonM,\mathfrak d)$.
An arrow $\mathfrak x\rightarrow \mathfrak y$ means that
ZFC proves $\mathfrak x\le \mathfrak y$.}
\end{figure}

\begin{figure}
  \centering
\[
\xymatrix@=3.5ex{
&&&&            \covN\ar[rdd] & \nonM \ar[rdd]      &  \mye\ar@{=}[d]\ar[ddr]      & \cofN\ar[r] &2^{\aleph_0} \\
&&&&                               & \mathfrak b\ar[u]  &  \mathfrak d &              \\
\aleph_1\ar[r] & \mfrak\ar[r] & \pfrak\ar[r]&\hfrak\ar[r]
& \addN\ar[uu] & \mye\ar@{=}[u] &  \covM\ar[u]& \nonN\ar[uu]
}
\]
    \caption{\label{fig:result}The model we construct in this paper; here $\mathfrak x\rightarrow \mathfrak y$ means that $\mathfrak x<\mathfrak y$ (when $\hfrak$ is omitted, any number  of the $<$ signs can be replaced by $=$ as desired).
    \protect\\
    This model corresponds to ``Version A'' (\ref{versionA}, Fig.~\ref{fig:1022}). We also realise another ordering of the Cicho\'{n} values, called ``Version B'' (\ref{versionB}, Fig.~\ref{fig:1131}).}
\end{figure}

\section{Preliminaries}\label{sec:prelim}
\subsection{The characteristics}
In addition to the Cicho\'n-characteristics
we will consider  the following ones, whose definitions are well known.
\begin{definition}\label{DefChar1}
    Let $\mathcal{P}$ be a class of forcing notions.
    \begin{enumerate}[(1)]
        \item $\mfrak(\mathcal{P})$ denotes the minimal cardinal where Martin's axiom for the posets in $\mathcal{P}$ fails. More explicitly, it is the minimal $\kappa$ such that, for some poset $Q\in\mathcal{P}$, there is a collection $\mathcal{D}$ of size $\kappa$ of dense subsets of $Q$ such that there is no filter in $Q$ intersecting all the members of $\mathcal{D}$.
        \item $\mfrak:=\mfrak(\textnormal{ccc})$.
        \item Write $a\subseteq^* b$ iff $a\smallsetminus b$ is finite. Say that $a\in[\omega]^{\aleph_0}$ is a \emph{pseudo-intersection} of $F\subseteq[\omega]^{\omega}$ if $a\subseteq^* b$ for all $b\in F$.
        \item The \emph{pseudo-intersection number $\pfrak$} is the smallest size of a filter base of a free filter on $\omega$ that has no pseudo-intersection in $[\omega]^{\aleph_0}$.
        \item The \emph{tower number} $\tfrak$ is the smallest order type of a $\subseteq^*$-decreasing sequence in $[\omega]^{\aleph_0}$ without pseudo-intersection.
        \item The \emph{distributivity number} $\hfrak$ is the smallest size of a collection of dense subsets of $([\omega]^{\aleph_0},\subseteq^*)$ whose intersection is empty.
        \item A family $D\subseteq[\omega]^{\aleph_0}$ is \emph{groupwise dense} if
        \begin{enumerate}[(i)]
            \item $a\subseteq^* b$ and $b\in D$ implies $a\in D$, and
            \item whenever $(I_n:n<\omega)$ is an interval partition of $\omega$, there is some $a\in[\omega]^{\aleph_0}$ such that $\bigcup_{n\in a}I_n\in D$.
        \end{enumerate}
        The \emph{groupwise density number $\gfrak$} is the smallest size of a collection of groupwise dense sets whose intersection is empty.
    \end{enumerate}
\end{definition}
It is well known 
that ZFC proves the following:
(See e.g.\ Blass~\cite{Blass}, and, for $\pfrak=\tfrak$, \cite{MSpt}, with Malliaris.)
\begin{equation}
    \label{eq:uqw523}
    \begin{gathered}
\mfrak\leq\pfrak=\tfrak\leq\hfrak\leq\gfrak,\quad \mfrak\leq\addN,\quad
\tfrak\leq\addM,\quad
\hfrak\leq\bfrak,\quad
\gfrak\leq\dfrak,\\
2^{<\tfrak}=\cfrak\text{ and }\cof(\cfrak)\geq\gfrak,
    \end{gathered}
\end{equation}
and all these cardinals are regular, with
the possible exception of $\mfrak$, $\dfrak$ and $\cfrak$.

\subsection{The old constructions}
In this paper, we will build on two
constructions from \cite{GKS,diegoetal} and \cite{KeShTa:1131}, which we call the ``old constructions'' and refer
to as \ref{versionA} and \ref{versionB}, respectively.
They force different values to several (or all) entries
of Cicho\'n's diagram.
We will not describe these constructions in detail, but refer to
the respective papers instead.

The ``basic versions'' of the constructions do not require
large cardinals and give
us different values for the ``left hand side'':

\begin{theorem}\label{oldnonba}
Assume that $\aleph_1\leq\lambda_1\leq\lambda_2\leq\lambda_3\leq\lambda_4$ are regular cardinals and $\lambda_4\leq\lambda_5\leq\lambda_6$.\smallskip

    \noindent\cite{diegoetal} If $\lambda_5$
    is a regular cardinal and $\lambda_6^{<\lambda_3}=\lambda_6$, then there is a f.s.\  iteration $\bar P^{\vA}$ of length of size $\lambda_6$ with cofinality $\lambda_4$, using iterands that are $(\sigma,k)$-linked for every $k\in\omega$, which forces
    \begin{multline}
    \tag{\vA}\label{versionAnb}
        \addN=\lambda_1,\ \covN=\lambda_2,\ \mathfrak{b}=\lambda_3,\ \nonM=\lambda_4,\\
        \covM=\lambda_5,\text{\ and } \mathfrak{d}=\mathfrak{c}=\lambda_6.
    \end{multline}

    \noindent\cite{KeShTa:1131} If $\lambda_5=\lambda_5^{<\lambda_4}$ and either $\lambda_2=\lambda_3$,\footnote{The result for the case $\lambda_2=\lambda_3$ is easily obtained with techniques from Brendle~\cite{Br}.} or $\lambda_3$ is $\aleph_1$-inaccessible,\footnote{A cardinal $\lambda$ is \emph{$\kappa$-inaccessible} if $\mu^\nu<\lambda$ for any $\mu<\lambda$ and $\nu<\kappa$.} $\lambda_2=\lambda_2^{<\lambda_2}$ and $\lambda_4^{\aleph_0}=\lambda_4$, then there is a f.s.\  iteration $\bar P^{\vB}$ of length of size $\lambda_5$ with cofinality $\lambda_4$, using iterands that are $(\sigma,k)$-linked for every $k\in\omega$, that forces
    \begin{multline}
    \tag{\vB}\label{versionBnb}
        \addN=\lambda_1,\ \mathfrak{b}=\lambda_2,\ \covN=\lambda_3,\ \nonM=\lambda_4,\\ \text{\ and }
        \covM=\mathfrak{c}=\lambda_5.
    \end{multline}
\end{theorem}

These consistency results correspond to $\lambda_1$--$\lambda_6$
of Figure~\ref{fig:1022}, and to $\lambda_1$--$\lambda_5$
of Figure~\ref{fig:1131}, respectively.

\begin{figure}
  \centering
  \begin{minipage}[t]{0.49\textwidth}
\[
\xymatrix@=3.5ex{
&            \lambda_2\ar[r]\ar@{.}[rd]|-{\kappa_7}        & \lambda_4 \ar[r]      &  \mye \ar[r]     & \lambda_8\ar[r] &\lambda_9 \\
&                               & \lambda_3\ar[r]\ar[u]  &  \lambda_6\ar[u] &              \\
  \aleph_1\ar[r]_{\kappa_9} & \lambda_1\ar[r]\ar[uu]^{\kappa_8} & \mye\ar[r]\ar[u] &  \lambda_5\ar[r]\ar[u]& \lambda_7\ar[uu]
}
\]
    \caption{\label{fig:1022}The \ref{versionA} order.}
  \end{minipage}
  \begin{minipage}[t]{0.49\textwidth}
\[
\xymatrix@=3.5ex{
&            \lambda_3\ar[r]^{\kappa_6}        & \lambda_4 \ar[r]      &  \mye \ar[r]     & \lambda_8\ar[r] &\lambda_9 \\
&                               & \lambda_2\ar[r]\ar[u]\ar@{.}[lu]|-{\kappa_7}  &  \lambda_7\ar[u] &              \\
  \aleph_1\ar[r]_{\kappa_9} & \lambda_1\ar[r]\ar[uu]\ar@{.}[ru]|-{\kappa_8} & \mye\ar[r]\ar[u] &  \lambda_5\ar[r]\ar[u]& \lambda_6\ar[uu]
}
\]
    \caption{\label{fig:1131}The \ref{versionB} order.}
  \end{minipage}
\end{figure}

\begin{remark}\label{rem:sing}
  Note that the hypothesis for \ref{versionBnb} is weaker than the hypothesis in the  original reference \cite{KeShTa:1131}, even more, GCH is not assumed at all. This
  strengthening is a result of simple modifications, which are presented in~\cite{modKST}. Moreover, note that $\dfrak$ can be singular in~\ref{versionAnb}, while $\covM$ can be singular in~\ref{versionBnb} (also in the left-side models from~\cite{GMS,GKS}).
\end{remark}

Both constructions can then be extended with Boolean
ultrapowers (more precisely:   compositions of finitely many
successive Boolean ultrapowers), to make all values simultaneously different:

\begin{theorem}\label{oldba}
    Assume $\aleph_1<\lambda_1<\lambda_2<\lambda_3\leq\lambda_4\leq\lambda_5\leq\lambda_6\leq\lambda_7\leq\lambda_8\leq\lambda_9$.\smallskip

    \noindent\cite{diegoetal} If $\aleph_1<\kappa_9<\lambda_1<\kappa_8<\lambda_2<\kappa_7<\lambda_3$ such that
    \begin{enumerate}[(i)]
        \item for $j=7,8,9$, $\kappa_j$ is strongly compact and $\lambda_j^{\kappa_j}=\lambda_j$,
        \item $\lambda_i$ is regular for $i\neq 6$ and
        \item $\lambda_6^{<\lambda_3}=\lambda_6$,
    \end{enumerate}
    then there is a f.s.\  ccc iteration $P^{\vA*}$
(a Boolean ultrapower of $P^\vA$) that forces the constellation of Figure~\ref{fig:1022}:
\begin{multline}
\tag{\vA*}\label{versionA}
    \addN=\lambda_1,\ \covN=\lambda_2,\ \mathfrak{b}=\lambda_3,\ \nonM=\lambda_4,\\
        \covM=\lambda_5,\ \mathfrak{d}=\lambda_6,\ \nonN=\lambda_7,\ \cofN=\lambda_8,\text{\ and }\mathfrak{c}=\lambda_9.
\end{multline}

   \noindent\cite{KeShTa:1131} If $\aleph_1<\kappa_9<\lambda_1<\kappa_8<\lambda_2<\kappa_7<\lambda_3<\kappa_6<\lambda_4$ such that
    \begin{enumerate}[(i)]
        \item for $j=6,7,8,9$, $\kappa_j$ is strongly compact and $\lambda_j^{\kappa_j}=\lambda_j$,
        \item $\lambda_i$ is regular for $i\neq 5$,
        \item $\lambda_2^{<\lambda_2}=\lambda_2$, $\lambda_4^{\aleph_0}=\lambda_4$, $\lambda_5^{<\lambda_4}=\lambda_5$, and
        \item $\lambda_3$ is $\aleph_1$-inaccessible,
    \end{enumerate}
    then there is a f.s.\  ccc iteration $P^{\vB*}$
(a Boolean ultrapower of $P^\vB$)
that forces the constellation of Figure~\ref{fig:1131}:
\begin{multline}
\tag{\vB*}\label{versionB}
    \addN=\lambda_1,\ \mathfrak{b}=\lambda_2,\ \covN=\lambda_3,\ \nonM=\lambda_4,\\
        \covM=\lambda_5,\ \nonN=\lambda_6,\ \mathfrak{d}=\lambda_7,\ \cofN=\lambda_8,\text{\ and }\mathfrak{c}=\lambda_9.
\end{multline}
\end{theorem}

More specifically: For $i=6,7,8,9$
let $j_i$
be a complete embedding associated with some suitable Boolean ultrapower from the completion of $\col(\kappa_i,\lambda_i)$, which yields $\crit(j_i)=\kappa_i$
and $\cf(j_i(\kappa_i))=|j_i(\kappa_i)|=\lambda_i$ (see a bit more details at the end of Subsection~\ref{subsec:blass}). Then
$P^{\vA*}=j_9(j_8(j_7(P^\vA)))$ forces the constellation of Figure~\ref{fig:1022}.
Analogously,
$P^{\vB*}=j_9(j_8(j_7(j_6(P^\vB))))$ forces
the constellation of Figure~\ref{fig:1131}.

\begin{remark}\label{rem:righteq}
   In the original results of Theorem~\ref{oldba}, all the inequalities are assumed to be strict (though in \ref{versionA} this is just from $\lambda_6$), but they can be equalities alternatively. Even more, whenever we change strict inequalities on the the right side to  equalities, we may weaken the assumption by requiring fewer strongly compact cardinals.  For example, if $\lambda_j=\lambda_{j+1}$ (for some $j=6,7,8,9$) then the compact cardinal $\kappa_j$ is not required, furthermore, the weaker assumption $\lambda_{9-j}\leq\lambda_{(9-j)+1}$ (for the dual cardinal characteristic, with $\lambda_0=\aleph_1$) is allowed in this case.

   Moreover, $\dfrak$ is allowed to be singular in~\ref{versionA}, while $\covM$ is allowed to be singular in~\ref{versionB} (likewise in the construction from~\cite{GKS}). See Remark~\ref{rem:sing}.
\end{remark}

\begin{notation}
\begin{enumerate}
    \item
Whenever we are investigating a
characteristic
$\mathfrak x$, we write $\lambda_{\mathfrak x}$ for the specific value we plan
to force
to it $\mathfrak x$.
For example, for~\ref{versionA}
$\lambda_2=\lambda_{\covN}$, whereas for~\ref{versionB}
$\lambda_2=\lambda_{\mathfrak b}$.
We remark that we \emph{do not} implicitly \emph{assume}
that
$P\Vdash \mathfrak x = \lambda_{\mathfrak x}$ for the $P$
under investigation; it is just an (implicit) declaration
of intent.


\item
Whenever we base an argument on one of the old constructions
above, and say ``we can modify the construction to additionally force\dots'', we implicitly
assume that the desired values $\lambda_{\mathfrak x}$
for the ``old'' characteristics satisfy
the assumptions we made in the ``old'' constructions
(such as ``$\lambda_{\mathfrak x}$ is regular'').
\end{enumerate}
\end{notation}

See~\cite[Subsec.~2.3]{GKMS1} for details on the history of the results of this section (and more).

\subsection{Blass-uniform cardinal characteristics, LCU and COB}\label{subsec:blass}


A more detailed discussion on the concepts reviewed in this subsection can be found in~\cite[Subsec.~2.1]{GKMS1}.

\begin{definition}[{\cite[Def.~2.1]{GKMS1}}]\label{def:blassu}
   A \emph{Blass-uniform cardinal characteristic} is a characteristic of the form
   \[\mathfrak d_R:=\min\{|D|:D\subseteq\omega^\omega\text{\ and }(\forall x\in\omega^\omega)\,(\exists y\in D) \ xRy\}\]
   for some Borel\footnote{More generally, it is just enough to assume that $R$ is absolute between the extensions we consider.} $R$.
%
%
\end{definition}

Its dual cardinal
\[{\mathfrak b}_R:=\min\{|F|:F\subseteq\omega^\omega\text{\ and }(\forall y\in\omega^\omega)\,(\exists x\in F) \ \neg xRy\}\]
is also Blass-uniform because $\mathfrak{b}_R=\mathfrak{d}_{R^\perp}$ where  $xR^\perp y$ iff $\neg(yRx)$.

In the practice, Blass-uniform cardinal characteristics are defined from a relation $R\subseteq X\times Y$ where $X$ and $Y$ are Polish spaces, but since we can translate such a relation to $\omega^\omega$ using Borel isomorphisms, it is enough to discuss relations on $\omega^\omega$.

Systematic research on such cardinal characteristics started  in the 1980s or possibly even earlier, see e.g.\ Fremlin~\cite{zbMATH03891346}, Blass~\cite{MR1234278, Blass} and
Vojt\'{a}\v{s}~\cite{MR1234291}.



%

\begin{example}
    The following are pairs
    of dual Blass-uniform cardinals
   $(\mathfrak b_R,\mathfrak d_R)$ for
    natural 
    Borel relations $R$:
    \begin{enumerate}[(1)]
        \item A cardinal on the left
        hand side of Cicho\'n's diagram and its dual
        on the right hand side:
        $(\addN,\cofN)$, $(\covN,\nonN)$,
        $(\addM,\cofM)$, $(\nonM,\covM)$, and $(\mathfrak b,\mathfrak d)$.
        \item $(\mathfrak s, \mathfrak r)=(\mathfrak b_R,\mathfrak d_R)$ where
        $\mathfrak s$ is the splitting number, $\mathfrak r$ is the reaping number,
        and $R$ is the relation on $[\omega]^{\aleph_0}$ defined by $xRy$ iff ``$x$ does not split $y$''.
    \end{enumerate}
\end{example}


\begin{definition}[{\cite[Def.~2.3]{GKMS1}}]\label{def:lcu.cob}
Fix a Borel relation $R$,
$\lambda$ a regular cardinal and
$\mu$ an arbitrary cardinal. We define two properties:

\begin{description}[labelindent=0pt] 
\item[Linearly cofinally unbounded]
    $\LCU_R(\lambda)$ means: There is a family
    $\bar f=(f_{\alpha}:\alpha<\lambda)$  of
    reals
    such that:
    \begin{equation}\label{eq:LCU}
    (\forall g\in \omega^\omega)\, (\exists\alpha\in\lambda)\,(\forall \beta\in \lambda\setminus \alpha) \ \lnot f_{\beta} R g.
    \end{equation}

\item[Cone of bounds] $\COB_R(\lambda,\mu)$ means:
    There is a $\mathord<\lambda$-directed partial order $\trianglelefteq$ on $\mu$,\footnote{I.e., every subset of $\mu$ of
    cardinality ${<}\lambda$ has a $\trianglelefteq$-upper bound}
    and a family $\bar g= (g_{s}:s\in \mu)$ of
    reals such that
    \begin{equation}\label{eq:COB}
    (\forall f\in\omega^\omega)\, (\exists s\in \mu)\, (\forall t\trianglerighteq s)\  f R g_{t}.
    \end{equation}
\end{description}
\end{definition}

\begin{fact}\label{fact:bla23424}
    $\LCU_R(\lambda)$ implies $\mathfrak b_R \le \lambda\leq\mathfrak d_R$.

    $\COB_R(\lambda,\mu)$
    implies $\mathfrak b_R \ge \lambda$ and
    $\mathfrak d_R \le \mu$.
\end{fact}



We often call
 the objects
$\bar f$
in the definition
of $\LCU$
and $(\trianglelefteq,\bar g)$ for $\COB$ ``strong witnesses'', 
and we say that the corresponding cardinal inequalities (or equalities)
are ``strongly witnessed''.
For example, ``$(\mathfrak b,\mathfrak d)=(\lambda_{\mathfrak b}, \lambda_{\mathfrak d})$ is strongly witnessed''
means: for the natural relation $R$ (namely, the relation $\le^*$ of eventual
dominance),
we have $\COB_R(\lambda_{\mathfrak b},\lambda_{\mathfrak d})$, $\LCU_R(\lambda_{\mathfrak b})$ and there is some regular $\lambda_0\leq\lambda_{\mathfrak{d}}$ such that $\LCU_R(\lambda)$ for all regular $\lambda\in[\lambda_0,\lambda_{\mathfrak{d}}]$ (this is to allow $\lambda_{\mathfrak{d}}$ to be singular as in \ref{versionAnb} and \ref{versionA} of Theorems~\ref{oldnonba} and \ref{oldba}).

\begin{remark}\label{RemStrWit}
The old constructions (\eqref{versionAnb}, \eqref{versionBnb} in Theorem~\ref{oldnonba}) use that we can first force strong witnesses
to the left hand side, and then preserve strong witnesses in
Boolean ultrapowers,
so that in the final model
all Cicho\'n-characteristics
are strongly witnessed.
In more detail,
for each dual pair $(\mathfrak x,\mathfrak y)$ in Cicho\'n's diagram,
there is a natural relation $R_{\mathfrak x}$ such that
$(\mathfrak x,\mathfrak y) = (\mathfrak b_{R_\mathfrak x},
\mathfrak d_{R_\mathfrak x})$.
We use these natural relations (with one exception\footnote{The exception is the following: In \ref{versionAnb}, for the pair $(\mathfrak{x},\mathfrak{y})=(\nonM,\covM)$ it is forced $\LCU_{\neq^*}(\lambda_4)$, $\LCU_{\neq^*}(\lambda_5)$ and $\COB_{\neq^*}(\lambda_4,\lambda_5)$ (here $x\neq^* y$ iff $x(i)\neq y(i)$ for all but finitely many $i$); in \ref{versionBnb},
for
$\mathfrak x=\covN$, we use the natural relation
$R_{\covN}$
(defined as the set of all pairs $(x,y)$ where the real $y$
is in the $F_\sigma$ set of full measure coded by $x$)
only for $\COB$.
In this version, we do not know
whether $P$ forces $\LCU_{R_{\covN}}(\lambda_\covN)$
(as we do not have sufficient preservation results for $R_\covN$, more specifically, we do know
whether $(\rho,\pi)$-linked posets are $R_\covN$-good.)
Instead,
we use another relation $R'$ (which defines
different, anti-localization characteristics $(\mathfrak b_{R'},
\mathfrak d_{R'})$),
for which ZFC proves $\covN\le \mathfrak b_{R'}$
    and $\nonN\ge \mathfrak d_{R'}$.
We can then show that $P$ forces  $\LCU_{R'}(\mu)$ for all
 regular $\lambda_{\covN}\le \mu\le |\delta|$.}) as follows: The initial forcing (without Boolean ultrapowers)
is a f.s.\ iteration $P$ of length $\delta$ and
forces $\LCU_{R_{\mathfrak x}}(\mu)$
for all regular $\lambda_{\mathfrak x}\le \mu \le |\delta|$,
and
$\COB_{R_{\mathfrak x}}(\lambda_{\mathfrak x},|\delta|)$.

Once we know that the initial forcing
$P$ gives strong witnesses for the desired
values $\lambda_{\mathfrak x}$ for all ``left-hand'' values
$\mathfrak x$ in Cicho\'n's diagram (and
continuum for the cardinals $\ge \mathfrak d,\nonN$ in ~\ref{versionA}
or $\ge\covM$ in~\ref{versionB}),
we use the following theorem to separate all the entries.
\end{remark}

\begin{theorem}[{\cite{KTT,GKS}}]\label{thm:fact:buppres}
Let $\nu<\kappa$ and $\lambda\ne\kappa$ be uncountable regular cardinals, $R$ a Borel relation, 
and let $P$ be a $\nu$-cc poset forcing that $\lambda$ is regular. Assume that $j:V\to M$ is an elementary embedding into a transitive class $M$ satisfying:
\begin{enumerate}[(i)]
  \item The critical point of $j$ is $\kappa$.
  \item $M$ is ${<}\kappa$-closed.\footnote{I.e., $M^{{<}\kappa}\subseteq M$.}
  \item For any cardinal $\theta>\kappa$ and any ${<}\theta$-directed partial order $I$,
       $j''I$ is cofinal in $j(I)$.
\end{enumerate}
Then:
\begin{enumerate}[(a)]
    \item $j(P)$ is a $\nu$-cc forcing.
    \item If $P\Vdash\LCU_R(\lambda)$, then
    $j(P)\Vdash\LCU_R(\lambda)$.
    \item If $\lambda<\kappa$ and $P\Vdash\COB_R(\lambda,\mu)$,
    then $j(P)\Vdash\COB_R(\lambda,|j(\mu)|)$.  
    \item If $\lambda>\kappa$ and $P\Vdash\COB_R(\lambda,\mu)$,
    then $j(P)\Vdash\COB_R(\lambda,\mu)$.  
\end{enumerate}
\end{theorem}
\begin{proof}
  We include the proof for completeness. Property (a) is immediate by (ii). First note that $j$ satisfies the following additional properties.
  \begin{enumerate}[(i)]
  \setcounter{enumi}{3}
    \item Whenever $a$ is a set of size ${<}\kappa$, $j(a)=j''a$.
    \item If $\cf(\alpha)\neq\kappa$ then $\cf(j(\alpha))=\cf(\alpha)$.
    \item If $\theta>\kappa$, $L$ is a set and $P\Vdash$``$(L,\dot \trianglelefteq)$ is ${<}\theta$-directed'' then $j(P)\Vdash$``$j''L$ is cofinal in $(j(L),j(\dot \trianglelefteq))$, and it is ${<}\theta$-directed''.
    \item $j(P)\Vdash$``$\cof(j(\lambda))=\lambda$''.
  \end{enumerate}
  Item (iv) follows from (i), and (v) follows from (iii). We show (vi). Let $L^*$ be the set of nice $P$-names of members of $L$, and order it by $\dot x\leq \dot y$ iff $P\Vdash\dot x\dot \trianglelefteq \dot y$. It is clear that $\leq$ is ${<}\theta$-directed on $L^*$. On the other hand, since any nice $j(P)$-name of a member of $j(L)$ is already in $M$ by (ii) and (a), $j(L^*)$ is equal to the set of nice $j(P)$-names of members of $j(L)$. Therefore, by (iii), $j''L^*$ is cofinal in $j(L^*)$. Note that $j''L^*$ is equal to the set of nice $j(P)$-names of members of $j''L$. Thus, (vi) follows.

  For (vii), the case $\lambda<\kappa$ is immediate by (i) and (ii); when $\lambda>\kappa$, apply (vi) to $(L,\dot \trianglelefteq)=(\lambda,\leq)$ (the usual order) and $\theta=\lambda$.

  To see (b), note that $M\vDash$``$j(P)\Vdash\LCU_R(j(\lambda))$'' and, by (a) and (ii), the same holds inside $V$ (because any nice name of an ordinal, represented by a maximal antichain on $P$, belongs to $M$, hence any nice name of a real, which in fact means that $j(P)\Vdash\LCU_R(\cof(j(\lambda)))$. By (vii) we are done.

  Now assume $P\Vdash\COB_R(\lambda,\mu)$ witnessed by $(\dot \trianglelefteq,\dot{\bar{g}})$. This implies $M\models$``$j(P)\Vdash(j(\dot \trianglelefteq),j(\dot{\bar{g}}))$ witnesses $\COB_R(j(\lambda),j(\mu))$''. If $\lambda<\kappa$ then $j(\lambda)=\lambda$ and it follows that $V\models$``$j(P)\Vdash\COB_R(\lambda,|j(\mu)|)$''. In the case $\lambda>\kappa$ apply (vi) to conclude that $j(P)$ forces that $(j(\dot{\bar{g}}(\beta)):\beta<\mu)$, with $j(\dot \trianglelefteq)$ restricted to $j''\mu$, witnesses $\COB_R(\lambda,\mu)$.
\end{proof}

If $\kappa$ is a strongly compact cardinal and $\theta^\kappa=\theta$, then there is an elementary embedding $j$ associated with a Boolean ultrapower of the completion of $\col(\kappa,\theta)$ such that $j$ satisfies (i)--(iii) of the preceding lemma and, in addition, for any cardinal $\lambda\geq\kappa$ such that either $\lambda\leq\theta$ or $\lambda^\kappa=\lambda$ holds,
we have $\max\{\lambda,\theta\}\leq j(\lambda)<\max\{\lambda,\theta\}^+$ (see details in \cite{KTT,GKS}).
Therefore, using only this lemma, it is easy to see how to get from
the old constructions (Theorem~\ref{oldnonba}) to the Boolean ultrapowers (Theorem~\ref{oldba}), as described in Remark~\ref{RemStrWit} (see details in \cite[Thm. 5.7]{diegoetal} for \ref{versionA} and \cite[Thm. 3.1]{KeShTa:1131} for \ref{versionB}).

\section{Cardinal characteristics in extensions without new \texorpdfstring{${<}\kappa$}{<kappa}-sequences}\label{sec:nnr}


This section summarizes the technical results introduced in~\cite{GKMS1}.

\begin{lemma}[{\cite[Lemma~3.1]{GKMS1}}]\label{lem:blassdistr}
Assume that $Q$ is $\theta$-cc and ${<}\kappa$-distributive
for $\kappa$ regular uncountable, and let $\lambda$ be a regular cardinal
and $R$ a Borel relation.
\begin{enumerate}
    \item If $\LCU_R(\lambda)$, then
    $Q\Vdash\LCU_R(\cof(\lambda))$.

    So if additionally $\lambda\le\kappa$ or $\theta\le\lambda$,
    then $Q\Vdash\LCU_R(\lambda)$.
    \item
    If $\COB_R(\lambda,\mu)$ and
    either $\lambda\le\kappa$ or $\theta\le\lambda$,
    then $Q\Vdash \COB_R(\lambda,|\mu|)$.

    So for any $\lambda$, $\COB_R(\lambda,\mu)$ implies
    $Q\Vdash \COB_R(\min(|\lambda|,\kappa),|\mu|)$.
\end{enumerate}
\end{lemma}

%


\begin{lemma}[{\cite[Lemma~3.2]{GKMS1}}]\label{lem:blasssub}
Assume that $R$ is a Borel relation, $P'$ is a complete
subforcing of $P$, $\lambda$ regular and $\mu$ is a cardinal,
both preserved in the $P$-extension.
\begin{enumerate}[(a)]
    \item If $P\Vdash\LCU_R(\lambda)$
witnessed by some $\dot{\bar{f}}$, and
 $\dot {\bar f}$ is actually a $P'$-name, then
 $P'\Vdash\LCU_R(\lambda)$.
 \item If $P\Vdash\COB_R(\lambda,\mu)$
witnessed by some $(\dot{\trianglelefteq},\dot {\bar g})$, and
$(\dot{\trianglelefteq},\dot {\bar g})$ is actually a $P'$-name, then
 $P'\Vdash\COB_R(\lambda,|\mu|)$.
\end{enumerate}
\end{lemma}

We now review three properties of cardinal
characteristics.

\begin{definition}[{\cite[Def.~3.3]{GKMS1}}]
Let $\mathfrak x$ be a cardinal characteristic.
\begin{enumerate}[(1)]
\item
   $\mathfrak x$ is \emph{\plike},
   if it has the following form:
   There is a formula $\psi(x)$ (possibly with, e.g., real parameters)
   absolute between universe extensions that do not add
   reals,\footnote{Concretely,
   if $M_1\subseteq M_2$ are transitive (possibly class) models of a fixed, large
   fragment of ZFC,
   with the same reals, then
   $\psi$ is absolute between $M_1$ and $M_2$.}
   such that $\mathfrak x$ is the smallest cardinality
   $\lambda$ of a set $A$ of reals such that $\psi(A)$.

   All Blass-uniform characteristics are \plike;
   other examples are $\mathfrak p$, $\mathfrak t$, $\mathfrak u$, $\mathfrak a$ and $\mathfrak i$.
\item
    $\mathfrak{x}$ is called \emph{\hlike},
    if it satisfies the same,
    but with $A$ being a family of sets of reals (instead of just a set
    of reals).

    Note that \plike\ implies \hlike, as we can include ``the family of
    sets of reals is a family of singletons'' in $\psi$.
    Examples are $\mathfrak h$ and $\mathfrak g$.
\item
   $\mathfrak{x}$ is called \emph{\mlike},
   if it has the following form:
   There is a formula $\varphi$ (possibly with, e.g., real parameters)
   such that $\mathfrak x$ is the smallest cardinality
   $\lambda$  such that $H({\le}\lambda)\vDash \varphi$.

   Any infinite  \plike\ characteristic is \mlike:  
   If $\psi$ witnesses \plike, then we can use
   $\varphi=(\exists A)\, [\psi(A)\&(\forall a\in A)\ a\text{ is a real}]$
   to get \mlike\ (since $H({\le}\lambda)$ contains all reals).
   Examples are\footnote{$\mathfrak m$ can be characterized as the smallest $\lambda$ such that there is in $H({\le} \lambda)$ a ccc forcing $Q$ and a family $\bar D$ of dense
   subsets of $Q$ such
    that ``there is no filter $F\subseteq Q$ meeting all $D_i$'' holds.}
   $\mathfrak m$, $\mathfrak m(\textrm{Knaster})$, etc.
\end{enumerate}
\end{definition}


\begin{lemma}[{\cite[Lemma~3.4]{GKMS1}}]\label{lem:trivial}
Let $V_1\subseteq V_2$ be models (possibly classes)
of set theory (or a sufficient fragment),
$V_2$ transitive and $V_1$ is either transitive or an elementary submodel of $H^{V_2}(\chi)$ for some large enough regular $\chi$,
such that $V_1\cap \omega^\omega = V_2\cap \omega^\omega$.
\begin{enumerate}[(a)]
    \item If $\mathfrak x$ is \hlike, then
    $V_1\vDash \mathfrak{x}=\lambda$ implies
    $V_2\vDash \mathfrak{x}\le |\lambda|$.
\end{enumerate}
In addition, whenever $\kappa$ is uncountable regular in $V_1$
and $V_1^{{<}\kappa}\cap V_2\subseteq V_1$:
\begin{enumerate}[(a)]
\setcounter{enumi}{1}
    \item
    If $\mathfrak x$ is \mlike,
    then
    $V_1\vDash \mathfrak{x}\ge\kappa$ iff
    $V_2\vDash \mathfrak{x}\ge \kappa$.
    \item If $\mathfrak x$ is \mlike\  and
    $\lambda<\kappa$, then
    $V_1\vDash \mathfrak{x}=\lambda$ iff
    $V_2\vDash \mathfrak{x}=\lambda$.
    \item If $\mathfrak x$ is \tlike\  and
    $\lambda=\kappa$, then
    $V_1\vDash \mathfrak{x}=\lambda$ implies
    $V_2\vDash \mathfrak{x}=\lambda$.
\end{enumerate}
\end{lemma}

We apply this to three situations:
Boolean ultrapowers, extensions by distributive forcings,
and complete subforcings:

\begin{corollary}[{\cite[Cor.~3.5]{GKMS1}}]\label{cor:trivial}
Assume that $\kappa$ is uncountable regular, $P\Vdash \mathfrak{x}=\lambda$, and
\begin{enumerate}[(i)]
    \item \underline{either} $Q$ is a $P$-name for a ${<}\kappa$-distributive
    forcing, and we set $P^+:=P*Q$
    and $j(\lambda):=\lambda$;
\item \underline{or} 
$P$ is $\nu$-cc for some $\nu<\kappa$,
$j:V\to M$ is a complete embedding into a transitive ${<}\kappa$-closed model $M$, $\crit(j)\geq\kappa$, and we set $P^+:=j(P)$,
\item\underline{or} $P$ is $\kappa$-cc, $M\preceq H(\chi)$ is
${<}\kappa$-closed, and we  set $P^+:= P\cap M$ and $j(\lambda):=|\lambda\cap M|$.
(So $P^+$ is a complete subposet of $P$; and if $\lambda\le\kappa$
then $j(\lambda)=\lambda$.)
\end{enumerate}
Then we get:
\begin{enumerate}[(a)]
    \item
    If $\mathfrak x$ is \mlike\ and
    $\lambda\ge\kappa$, then
    $P^+\Vdash \mathfrak{x}\ge \kappa$.
    \item If $\mathfrak x$ is \mlike\  and
    $\lambda<\kappa$, then
    $P^+\Vdash \mathfrak{x}=\lambda$.
    \item  If $\mathfrak x$ is \hlike\ then $P^+\Vdash\mathfrak{x}\le |j(\lambda)|$. Concretely,
    \begin{itemize}
        \item[{}] for (i):
        $P^+\Vdash \mathfrak{x}\le |\lambda|$;
        \item[{}] for (ii):
        $P^+\Vdash \mathfrak{x}\le |j(\lambda)|$;
        \item[{}]
        for (iii):
        $P^+\Vdash \mathfrak{x}\le |\lambda\cap M|$.
    \end{itemize}
    \item So if $\mathfrak x$ is \tlike\ and $\lambda=\kappa$,
    then for (i) and (iii) we get $P^+\vDash \xfrak=\kappa$.
\end{enumerate}
\end{corollary}

\subsection{On the role of large cardinals in our construction}\label{ss:heike}
It is known that NNR extensions will preserve Blass-uniform characteristics in the absence of at least some large cardinals.
More specifically:

\begin{lemma}\label{lem:heike}
  Assume that $0^\#$ does not exist. Let
  $V_1\subseteq V_2$ be transitive class
  models with the same reals, and assume
  $V_1\models \mathfrak{x}=\lambda$ for some Blass-uniform $\mathfrak x$.
  Then $V_2\models \mathfrak x=|\lambda|$.
\end{lemma}
(This is inspired by the deeper observation  of Mildenberger~\cite[Prop.~2.1]{MR1625907}, who uses
the  covering lemma~\cite{MR667224}
for the Dodd-Jensen core model to show that in \emph{cardinality preserving}
NNR extensions, a measurable in an inner model
is required to change the value of a Blass-uniform characteristic.)

\begin{proof}
  Fix a bijection in $V_1$ between the reals and some ordinal $\alpha $.
  Assume that in $V_2$, $X\subseteq \omega^\omega$ witnesses
  that $\aleph_1\le \mathfrak x\le \mu<|\lambda|$.
  Using the bijection, we can interpret $X$ as a subset of $\alpha$.
  According to Jensen's covering lemma  in $V_2$, there is in $L$ (and thus in $V_1$)
  some $X'\supseteq X$ such that $|X'|=|X|$ in $V_2$, in particular
  $|X'|^{V_2}<\lambda$.
  Therefore, $|X'|^{V_1}<\lambda$ as well;
  and, by absoluteness, $V_1$ thinks that
   $X'$ witnesses $\mathfrak x<\lambda$, a contradiction.
\end{proof}

Recall the ``old'' Boolean ultrapower constructions~\ref{versionA}:
Assume that we start with a forcing notion
$P$ forcing $\mathfrak d=2^{\aleph_0}=\lambda_6$.
We now use the
elementary embedding $j=j_7:V\to M$ with critical point $\kappa_7$,
and set $P':=j(P)$. As we have seen, $P'$ still forces
$\mathfrak d=\lambda_6$, but $2^{\aleph_0}=\lambda_7 = |j(\kappa_7)|$.

So let $G$ be a $P'$-generic filter over $V$
(which is also $M$-generic). Set $V_1:=M[G]$ and
$V_2:=V[G]$. Then $V_1$ is a ${<}\kappa$-complete submodel of
$V_2$. By elementaricity,
$M\models j(P)\Vdash \mathfrak d=j(\lambda_6)$.
So $V_1\models \mathfrak d=j(\lambda_6)$,
whereas $V_2\models \mathfrak d=\lambda_6<|j(\lambda_6)|$.

Hence, for this specific constellation of models, some large cardinals (at least $0^\#$) are
required (for our construction we actually use strongly compact cardinals).

\section{Applications}\label{sec:appl}

For notation simplicity, we declare that ``$1$-Knaster" means ``ccc", and ``$\omega$-Knaster" means ``precaliber $\aleph_1$".
Corollary~\ref{cor:trivial} gives us
11 characteristics:

\begin{lemma}\label{11values}
    Given $\aleph_1\leq\lambda_{\mfrak}<\kappa_9$ regular and $1\leq k_0\leq\omega$,
    we can modify $P^{\vA*}$ (and also $P^{\vB*}$)
    so that we additionally force:
    \begin{enumerate}[(1)]
      \item $\mfrak(k\textrm{-Knaster})=\aleph_1$ for $1\le k<k_0$,
      \item $\mfrak(k\textrm{-Knaster})=\lambda_{\mfrak}$ for $k\geq k_0$,
      \item $\pfrak \ge \kappa_9$.
    \end{enumerate}
\end{lemma}
\begin{proof}
    As in~\cite[Sect.~4 \&~5]{GKMS1},
    we can modify $P^{\vA}$ to construct a ccc poset $P'$ forcing the same as $P^{\vA}$ and, in addition, $\pfrak=\bfrak$, and both (1) and (2). Apply Boolean ultrapowers to $P'$ as in the ``old" construction, resulting in $P^*$.
    We can apply  Corollary~\ref{cor:trivial}(ii), more specifically the consequences (a) and (b): (b) implies that
    $P^*$ forces (1) and (2),
    while (a) implies that $P^*$ forces $\pfrak\geq\kappa_9$.
    And just as in the ``old'' construction,
    we can use Theorem~\ref{thm:fact:buppres}
    to show that
    $P^*$ forces the desired values to the Cicho\'n-characteristics.
\end{proof}

The following lemma is useful to modify $\gfrak$ and $\cfrak$ via complete subposets, while preserving \mlike\ and Blass-uniform values from the original poset.


\begin{lemma}[{\cite[Lemma~6.3]{GKMS1}}]\label{lem:subforcing}
Assume the following:
\begin{enumerate}[(1)]
    \item $\aleph_1\le\kappa\le\nu\le \mu$,  where $\kappa$ and $\nu$ are regular
and $\mu=\mu^{<\kappa}\geq\nu$,
    \item $P$ is a $\kappa$-cc poset forcing $\cfrak>\mu$.  
\item For some Borel relations $R^1_i$ ($i\in I_1$) on $\omega^\omega$ and some regular $\lambda^1_i\leq\mu$:
$P$ forces $\LCU_{R^1_i}(\lambda^1_i)$

\item For some Borel relations $R^2_i$ ($i\in I_2$) on $\omega^\omega$, $\lambda^2_i\leq\mu$ regular and a cardinal $\vartheta^2_i\leq\mu$: $P$ forces  $\COB_{R^2_i}(\lambda^2_i,\vartheta^2_i)$.

\item For some \mlike\ characteristics $\mathfrak y_j$ ($j\in J$)
and $\lambda_j<\kappa$:
$P\Vdash \mathfrak y_j=\lambda_j$.

\item For some \mlike\ characteristics $\mathfrak y'_k$ ($k\in K$): $P\Vdash \mathfrak{y}'_k\geq\kappa$.

\item $|I_1\cup I_2\cup J\cup K|\leq\mu$.
\end{enumerate}
Then there is
a complete subforcing $P'$ of $P$
of size
$\mu$
forcing
\begin{enumerate}[(a)]
    \item $\mathfrak y_j=\lambda_j$, $\mathfrak{y}'_k\geq\kappa$, $\LCU_{R^1_i}(\lambda^1_i)$ and $\COB_{R^2_{i'}}(\lambda^2_{i'},\vartheta^2_{i'})$ for all $i\in I_1$, $i'\in I_2$, $j\in J$ and $k\in K$;
    \item $\cfrak=\mu$ and $\gfrak\leq\nu$.
\end{enumerate}
\end{lemma}

\begin{remark}\label{remCOB}
   So we can preserve $\COB_R(\lambda,\theta)$ provided both
   $\lambda$ and $\theta$ are $\le\mu$.

   For larger $\lambda$ or $\theta$ this is generally not possible.
   E.g., if $\lambda>\mu$, then $\COB_R(\lambda,\theta)$ will fail in the $P'$-extension as
   it implies $\bfrak_R\ge\lambda>\mu=\cfrak$; in fact, according to~\cite[Lemma~1.6]{GKMS2},
   we actually get  $P'\Vdash\COB_R(\nu,\nu)$.
   However we do get the following (the proof is straightforward):


   If we assume, in addition to the conditions of Lemma~\ref{lem:subforcing}, that $\mu^{<\nu}=\mu$ and
   $P\Vdash \COB_{R}(\lambda,\theta)$
   for some Borel relation $R$ and
      $\lambda\le\nu$ (now allowing also $\theta>\mu$),
   then we can construct $P'$ such that
   $P'\Vdash\COB_{R}(\lambda,\mu)$.
   (But this only gives us $\dfrak_{R}\le\mu=\cfrak$.)

   In the case of $\nu<\lambda$, 
   we  have $P\Vdash \COB_{R}(\nu,\theta)$, so as we have just seen we can get $P'\Vdash \COB_{R}(\nu,\min(\mu,\theta))$
   (which implies $\mathfrak b_R\ge \nu$, which is a bit better
   than the $\mathfrak b_R\ge \kappa$ we get from (6)).
\end{remark}

The following two results deal with $\pfrak$.

\begin{lemma}[{\cite[Lemma~7.2]{GKMS1}}]\label{lem:lqhwo5u25}
Assume $\xi^{<\xi}=\xi$, $P$ is $\xi$-cc,
and set $Q=\xi^{<\xi}$ (ordered by extension).
Then $P$ forces that $Q^V$ preserves all cardinals
and cofinalities.
Assume  $P\Vdash\mathfrak x=\lambda$ (in particular that
$\lambda$ is a cardinal), and let $R$ be a Borel relation.
\begin{enumerate}[(a)]
    \item If $\mathfrak x$ is \mlike:
    $\lambda<\xi$ implies $P\times Q\Vdash\mathfrak x=\lambda$;
    $\lambda\ge\xi$ implies
    $P\times Q\Vdash \mathfrak x\ge \xi$.
    \item If $\mathfrak x$ is \hlike:
    $P\times Q\Vdash\mathfrak x\le \lambda$.
    \item $P\Vdash \LCU_R(\lambda)$ implies
    $P\times Q\Vdash \LCU_R(\lambda)$.
    \item $P\Vdash \COB_R(\lambda,\mu)$ implies
    $P\times Q\Vdash \COB_R(\lambda,\mu)$.
\end{enumerate}
\end{lemma}


\begin{lemma}[{\cite{longlow}}, {\cite[Lemma~7.3]{GKMS1}}]\label{fact:wreqwr}
Assume that $\xi=\xi^{<\xi}$ and $P$ is a $\xi$-cc poset that forces $\xi\le \mathfrak p$.
In the $P$-extension $V'$, let $Q=(\xi^{<\xi})^V$. Then,
\begin{enumerate}[(a)]
    \item $P\times Q=P*Q$ forces
$\mathfrak p=\xi$
    \item If in addition $P$ forces $\xi\leq\pfrak=\mathfrak{h}=\kappa$ then $P\times Q$ forces $\mathfrak{h}=\kappa$.
\end{enumerate}
\end{lemma}

We are now ready to prove the consistency of 13 pairwise different classical characteristics. Note that the following result allows $\dfrak$ and $\cfrak$ singular. A similar result with $\covM$ and $\cfrak$ singular can be obtained if we base the initial construction in~\cite{GKS} instead of~\cite{diegoetal}.

\begin{theorem}\label{thm:bla66}
   Assume $\aleph_1\leq\lambda_{\mfrak}\leq\lambda_{\pfrak}\leq\lambda_{\mathfrak{h}}\leq\kappa_9<\lambda_1<\kappa_8<\lambda_2<\kappa_7<\lambda_3\leq\lambda_4\leq\lambda_5\leq\lambda_6\leq\lambda_7\leq\lambda_8\leq\mu$ such that
   \begin{enumerate}[(i)]
       \item For $j=7,8,9$, $\kappa_i$ is strongly compact,
       \item $\lambda_j^{\kappa_j}=\lambda_j$ for $j=7,8$,
       \item $\lambda_i$ is regular for $i\neq6$
       \item $\lambda_{\pfrak}^{<\lambda_{\pfrak}}=\lambda_{\pfrak}$
       \item $\lambda_6^{<\lambda_3}=\lambda_6$,
       \item $\mu^{<\lambda_{\hfrak}}=\mu$.
   \end{enumerate}
   Then there is a $\lambda_\pfrak^+$-cc poset $P$ which preserves cofinalities and forces (1) and (2) of Lemma~\ref{11values}, and
   \begin{multline*}
    \pfrak=\lambda_{\pfrak},\ \mathfrak{h}=\mathfrak{g}=\lambda_{\mathfrak{h}},\ \addN=\lambda_1,\ \covN=\lambda_2,\ \mathfrak{b}=\lambda_3,\ \nonM=\lambda_4,\\
        \covM=\lambda_5,\ \mathfrak{d}=\lambda_6,\ \nonN=\lambda_7,\ \cofN=\lambda_8,\text{\ and }\mathfrak{c}=\mu.
   \end{multline*}
\end{theorem}
\begin{proof}
    Let $P^*$ be the ccc poset obtained in the proof of Lemma~\ref{11values} for $\lambda_9:=(\mu^{\kappa_9})^+$ (the modification of $P^{\vA*}$). This is a ccc poset of size $\lambda_9$ that forces the values of the Cicho\'n-characteristics as in Theorem~\ref{oldba} (\ref{versionA}) with strong witnesses, and forces (1) and (2) of Lemma~\ref{11values} and $\pfrak\geq\kappa_9$ whenever $\lambda_{\mfrak}<\kappa_9$, but in the case $\lambda_{\mfrak}=\kappa_9$ it forces $\mfrak(k_0\text{-Knaster})\geq\kappa_9$ instead of (2).

    By application of Lemma~\ref{lem:subforcing} to $\kappa=\nu=\lambda_{\mathfrak{h}}$ and to $\mu$, we find a complete subposet $P'$ of $P^*$ forcing (1) and (2) of Lemma~\ref{11values}, $\lambda_{\mathfrak{h}}\leq\pfrak\leq\mathfrak{g}\leq\lambda_{\mathfrak{h}}$ (so they are equalities), $\cfrak=\mu$ and that the values of the other cardinals in Cicho\'n's diagram are the same values forced by $P^*$, even with strong witnesses. This is clear in the case $\lambda_{\mfrak}<\lambda_{\mathfrak{h}}$, but the case $\lambda_{\mfrak}=\lambda_{\mathfrak{h}}$ (even $\lambda_{\mfrak}=\kappa_9$) is also fine because $P'$ would force $\lambda_{\mfrak}\leq\mfrak(k_0\text{-Knaster})\leq\mfrak(\textnormal{precaliber})\leq\pfrak\leq\mathfrak{g}\leq\lambda_{\mfrak}$.

    If $\lambda_{\pfrak}=\lambda_{\mathfrak{h}}$ then we would be done, so assume that $\lambda_{\pfrak}<\lambda_{\mathfrak{h}}$. Hence, by  Lemmas~\ref{lem:lqhwo5u25} and~\ref{fact:wreqwr}, $P:=P'\times(\lambda_{\pfrak}^{<\lambda_{\pfrak}})$ is as required. It is clear that $P$ forces $\mfrak(k_0\text{-Knaster})=\mfrak(\textnormal{precaliber})=\lambda_{\mfrak}$ when $\lambda_{\mfrak}<\lambda_{\pfrak}$, but the same happens when $\lambda_{\mfrak}=\lambda_{\pfrak}$ because $P$ would force $\lambda_{\mfrak}\leq\mfrak(k_0\text{-Knaster})\leq\mfrak(\textnormal{precaliber})\leq\pfrak\leq\lambda_{\mfrak}$.
\end{proof}

The same argument works to get a similar version of the previous result for \ref{versionB} where $\covM$ and $\cfrak$ are allowed to be singular.

\begin{theorem}
Assume $\aleph_1\leq\lambda_{\mfrak}\leq\lambda_{\pfrak}\leq\lambda_{\mathfrak{h}}<\kappa_9<\lambda_1<\kappa_8<\lambda_2<\kappa_7<\lambda_3<\kappa_6<\lambda_4\leq\lambda_5\leq\lambda_6\leq\lambda_7\leq\lambda_8\leq\mu$ such that
\begin{enumerate}[(i)]
        \item for $j=6,7,8,9$, $\kappa_j$ is strongly compact,
        \item $\lambda_j^{\kappa_j}=\lambda_j$ for $j=6,7,8$,
        \item $\lambda_i$ is regular for $i\neq 5$,
        \item $\lambda_{\pfrak}^{<\lambda_{\pfrak}}=\lambda_{\pfrak}$
        \item $\lambda_2^{<\lambda_2}=\lambda_2$, $\lambda_4^{\aleph_0}=\lambda_4$, $\lambda_5^{<\lambda_4}=\lambda_5$,
        \item $\lambda_3$ is $\aleph_1$-inaccessible, and
        \item $\mu^{<\lambda_{\hfrak}}=\mu$.
    \end{enumerate}
    Then there is a  $\lambda_\pfrak^+$-cc poset $P$, preserving cofinalities, that forces (1) and (2) of Lemma~\ref{11values}, and
   \begin{multline*}
    \pfrak=\lambda_{\pfrak},\ \mathfrak{h}=\mathfrak{g}=\lambda_{\mathfrak{h}},\ \addN=\lambda_1,\ \mathfrak{b}=\lambda_2,\ \covN=\lambda_3,\ \nonM=\lambda_4,\\
        \covM=\lambda_5,\ \nonN=\lambda_6,\ \mathfrak{d}=\lambda_7,\  \cofN=\lambda_8,\text{\ and }\mathfrak{c}=\mu.
   \end{multline*}
\end{theorem}

In the previous proof we can preserve all characteristics only because, before applying Lemma~\ref{lem:subforcing}, $\cofN$ (equal to $\lambda_8$) is smaller than the continuum (equal to $\lambda_9$).
In particular, if we use version~\ref{versionAnb} without large cardinals, and we cannot further increase the continuum above $\cofN$, then the methods of this section only ensure a model of (1) and (2) of Lemma~\ref{11values} plus
\begin{multline*}
  \pfrak=\lambda_\pfrak,\ \gfrak=\hfrak=\lambda_\hfrak,\\ \min\{\lambda_1,\lambda_\hfrak\}\leq\addN\leq\lambda_1,\ \min\{\lambda_2,\lambda_\hfrak\}\leq\covN\leq\lambda_2,\ \min\{\lambda_3,\lambda_\hfrak\}\leq\bfrak\leq\lambda_3,\\
  \nonM=\lambda_4,\  \covM=\lambda_5,\
  \dfrak=\nonN=\cfrak=\mu,
\end{multline*}
whenever $\aleph_1\leq\lambda_\mfrak\leq\lambda_\pfrak=\lambda_{\pfrak}^{<\lambda_\pfrak}
\leq\lambda_\hfrak\leq\lambda_3$ are regular, $\lambda_\mfrak\leq\lambda_1$ and $\lambda_5\leq\mu=\mu^{<\lambda_3}<\lambda_6$, where $\lambda_i$ (i=1,\ldots,6) are as in \ref{versionAnb} (see Remark \ref{remCOB}). That is, some left side Cicho\'n-characteristics do not get decided unless $\lambda_\xfrak\leq\lambda_\hfrak$ Hence, it is unclear whether $\hfrak$ gets separated from all the left side characteristics.
A similar situation occurs with version~\ref{versionBnb}: we may lose
$\mathfrak x\ge \lambda_{\mathfrak x}$
for any left side characteristic $\mathfrak x$ when $\lambda_\hfrak<\lambda_\xfrak$.

\section{Reducing gaps (or getting rid of them)}\label{sec:coll}

We start with the following well-known result.

\begin{lemma}[Easton's lemma]
    Let $\kappa$ be an uncountable cardinal, $P$ a $\kappa$-cc poset and let $Q$ be a ${<}\kappa$-closed poset. Then $P$ forces that $Q$ is ${<}\kappa$-distributive.
\end{lemma}
\begin{proof}
   For successor cardinals, this is proved in~\cite[Lemma 15.19]{Je03},
   but the same argument is valid for any regular cardinal. Singular cardinals are also fine because, for $\kappa$ singular, ${<}\kappa$-closed implies ${<}\kappa^+$-closed.
\end{proof}

As mentioned in Remark \ref{rem:righteq}, we can choose
right side Cicho\'n-characteristics rather arbitrarily or
even choose them to be equal
(equality allows a construction from fewer
compact cardinals). However, large gaps were required between some
left side cardinals. We deal with this problem now, and show that
we can reasonably assign arbitrary values to all characteristics,
and in particular set any ``reasonable selection'' of them equal.

Let us introduce notation to describe this effect:
\begin{definition}\label{def:lecons}
   Let $\bar{\xfrak}=(\xfrak_i:i<n)$ be a finite sequence of cardinal characteristics (i.e., of definitions). Say that $\bar{\xfrak}$ is a $<$-\emph{consistent sequence} if the statement $\xfrak_0<\ldots<\xfrak_{n-1}$ is consistent with ZFC (perhaps modulo large cardinals).

   A consistent sequence $\bar{\xfrak}$ is $\le$-\emph{consistent} if, in the previous chain of inequalities, it is consistent to replace any desired instance of $<$ with $=$. More formally, for any interval partition $(I_k:k<m)$ of $\{0,\ldots,n-1\}$, it is consistent that $\xfrak_i=\xfrak_j$ for any $i,j\in I_k$, and $\xfrak_i<\xfrak_j$ whenever $i\in I_k$, $j\in I_{k'}$ and $k<k'<m$.
\end{definition}

For example, the sequence
\[(\aleph_1,\addN,\covM,\bfrak,\nonM,\covM,\dfrak)\]
is $\le$-consistent, as well as
\[(\aleph_1,\addN,\bfrak,\covN,\nonM,\covM),\]
see Theorem~\ref{oldnonba}. Previously, it
had not been known whether the sequences of ten Cicho\'n-characteristics from \cite{GKS,diegoetal,KeShTa:1131} are $\le$-consistent:
It is not immediate that cardinals on the left side can be equal while separating everything on the right side. The reason is that, to separate cardinals on the right side, it is necessary to have a strongly compact cardinal between the dual pair of cardinals on the left, thus the left side gets separated as well. But thanks to the collapsing method of this section,
we can equalize cardinals on the left as well. As a result, we obtain the following:\footnote{Each sequence yields $2^{11}$ many consistency results (not all of them new, obviously; CH is one of them).}

\begin{lemma}\label{lem:wqr33}
The sequences
\begin{align*}
  (\aleph_1,\mfrak,\pfrak,\addN,\covM,\bfrak,\nonM,\covM,\dfrak,\nonN,\cofN,\cfrak) & \text{and} \\
  (\aleph_1,\mfrak,\pfrak, \addN,\bfrak,\covN,\nonM,\covM,\nonN,\dfrak,\cofN,\cfrak) &
\end{align*}
are $\le$-consistent (modulo large cardinals).
\end{lemma}

(Note that we lose $\hfrak$ in the process.)

To prove this claim, we use the following:


\begin{assumption}\label{asm:prelim}
\begin{enumerate}
    \item $\kappa$ is regular uncountable.
    \item $\theta\ge\kappa$, $\theta=\theta^{<\kappa}$.
    \item $P$ is $\kappa$-cc
        and forces that $\mathfrak x=\lambda$
        for some characteristic $\mathfrak x$ (so in particular
        $\lambda$ is a cardinal in the $P$-extension).
    \item $Q$ is $\mathord<{\kappa}$-closed.
    \item\label{item:ccwr} $P\Vdash Q\text{ is }\theta^+\textnormal{-cc}$.\footnote{I.e., $P$ forces that all antichains of $Q$ have size ${\le}\theta$.}
    \item We set $P^+ := P\times Q=P*Q$.
    We call the $P^+$-extension $V''$ and the intermediate $P$-extension $V'$.
\end{enumerate}
\end{assumption}
(We will actually have $|Q|=\theta$, which implies~(\ref{item:ccwr})).


Let us list a few simple facts:
\begin{enumerate}[({P}1)]
    \item In $V'$, all $V$-cardinals $ {\ge}\kappa$ are still cardinals, and
        $Q$ is a $\mathord<\kappa$-distributive forcing
        (due to Easton's lemma).
        So we can apply Lemma~\ref{lem:blassdistr} and Corollary~\ref{cor:trivial}.
    \item
        Let $\mu$ be the successor (in $V$ or equivalently in $V'$)
        of $\theta$.
        So in $V'$, $Q$ is $\mu$-cc and
        preserves all cardinals $\le \kappa $ as well as all cardinals
        $\ge \mu$.
   \item So if $V\models$``$\kappa\le \nu\le \theta$'',
   then in $V''$, $\kappa\le |\nu|<\mu$. The $V''$ successor of $\kappa$
   is $\le \mu$.
\end{enumerate}


We now apply it to a collapse:

\begin{lemma}\label{lem:collapse}
Let $R$ be a Borel relation, $\kappa$  be regular,
$\theta>\kappa$, $\theta^{<\kappa}=\theta$,
$P$ $\kappa$-cc, and set
$Q:=\col(\kappa,\theta)$,
i.e., the set of partial functions $f:\kappa \to \theta$
of size $<\kappa$. Then:
\begin{enumerate}[(a)]
    \item $P\times Q$ forces $|\theta|=\kappa$.

    \item If $P$ forces that  $\lambda$  is a cardinal then
\[P\times Q \Vdash|\lambda|=
       \begin{cases}
         \kappa  &\text{if (in $V$)  $\kappa\le \lambda\le \theta$}\\
         \lambda &\text{otherwise.}
       \end{cases}
\]

    \item If $\mathfrak{x}$ is \mlike, $\lambda<\kappa$ and $P\Vdash \mathfrak{x}=\lambda$, then $P\times Q\Vdash\mathfrak x=\lambda$.
    \item If $\mathfrak{x}$ is \mlike\ and $P\Vdash \mathfrak{x}\geq\kappa$, then $P\times Q\Vdash\mathfrak x\geq\kappa$.
    \item If $R$ is a Borel relation 
    then
       \begin{enumerate}[(i)]
           \item $P\Vdash$``$\lambda$ regular and $\LCU_R(\lambda)$'' implies $P\times Q\Vdash\LCU_R(|\lambda|)$.
           \item $P\Vdash$``$\lambda$ is regular and $\COB_R(\lambda,\mu)$'' implies $P\times Q\Vdash\COB_R(|\lambda|,|\mu|)$.
       \end{enumerate}
\end{enumerate}
\end{lemma}

\begin{proof}
    As mentioned, Assumption~\ref{asm:prelim} is met; in particular,
    $P$ forces that $\check Q$ is ${<}\kappa$-distributive (by \ref{asm:prelim}(P2)), so we can use
    Lemma~\ref{lem:blassdistr} and Corollary~\ref{cor:trivial}. Also note that, whenever $\kappa<\lambda\leq\theta$ and $P\Vdash$``$\lambda$ is regular'', $P\times Q$ forces $\cof(\lambda)=\kappa=|\lambda|$.
\end{proof}


So we can start, e.g., with a
forcing $P_0$ as in  Theorem~\ref{thm:bla66}:
$P_0$ is $\lambda_\pfrak^+$-cc, and forces
strictly increasing values to the characteristics
in the first, say, sequence of Lemma~\ref{lem:wqr33}.

We now pick some $\kappa_0<\theta_0$,
satisfying
$\lambda_\pfrak< \kappa_0$ and the assumptions of the previous Lemma, i.e.,
$\kappa_0$ is regular and $\theta_0^{{<}\kappa_0}=\theta_0$.
Let $Q_0$ be the collapse of $\theta_0$ to $\kappa_0$, a forcing of size $\theta_0$. So $P_1:=P_0\times Q_0$ is $\theta^+_0$-cc and,
according to the previous Lemma,
still forces the ``same'' values (and in fact strong witnesses)
to the Cicho\'n-characteristics
(including the case that any value $\lambda_i$ with
$\kappa_0<\lambda_i\le\theta_0$ is collapsed to $|\lambda_i|=\kappa_0$).
The \mlike\ invariants below $\kappa_0$,
i.e., $\mfrak$ and $\pfrak$, are also unchanged.

We now pick another pair $\theta_0<\kappa_1<\theta_1$
(with the same requirements) and take the product of $P_1$
with the collapse $Q_1$ of $\theta_1$ to $\kappa_1$, etc.

In the end, we get $P_0\times Q_0\times \cdots\times Q_n$.
Each characteristic which by $P$ was forced to have value $\lambda$
now is forced to have value $|\lambda|$, which is
$\kappa_m$ if $\kappa_m\le \lambda\le \theta_m$ for some $m$,
and $\lambda$ otherwise.
This immediately gives the

\begin{proof}[Proof of Lemma~\ref{lem:wqr33}]
We start with GCH, and construct the inital forcing to already result in
the desired (in)equalities between $\aleph_1,\mfrak,\pfrak$
and to result in pairwise different regular  Cicho\'n values $\lambda_i$ and $\pfrak<\addN$.

Let $(I_m)_{m\in M}$ be the interval partition of the sequence
$(\pfrak,\addN,\dots,\cfrak)$ indicating which characteristics we want
to identify.
For each non-singleton $I_m$, let $\kappa_m$ be the value
of the smallest characteristic in $I_m$, and $\theta_m$ the largest.
Note that $\theta_m<\kappa_{m+1}<\theta_{m+1}$.
Then $P_0\times Q_0\times \cdots\times Q_{M-1}$ forces
that all characteristics in $I_m$ have value $\kappa_m$, as desired.
\end{proof}

Similarly and easily we get the following:
\begin{lemma}\label{lem:uihwqer5}
We can assign the values $\aleph_1,
\aleph_2,\dots,\aleph_{12}$
to the first sequence of Lemma~\ref{lem:wqr33}
(as in Figure~\ref{fig:increasing}).

We can do the same for the second sequence.
\end{lemma}

\begin{proof}
Again, start with GCH and $P_0$ forcing the desired values for
$\mfrak$ and $\pfrak$ (now $\aleph_2$ and $\aleph_3$)
and pairwise distinct
regular Cicho\'n values $\lambda_i$.
Then pick $\kappa_0=\lambda_\pfrak^+=\aleph_4$
and $\theta_0=\lambda_1$ (which then
becomes $\aleph_4$ after the collapse).
Then set $\kappa_1=\lambda_1^+$ (which would be $\aleph_5$ after the
first collapse), and $\theta_1=\lambda_2$, etc.
\end{proof}


We can of course just as well assign the values $(\aleph_{\omega\cdot m+1})_{1\le m\le 12}$ instead of $(\aleph_m)_{1\le m\le 12}$. It is a bit awkward to
make precise the (not entirely correct) claim ``we can assign whatever we want''; nevertheless we will do just that in the rest of this section.


\begin{theorem}\label{nogapsAsimple}
Assume GCH. Let $1\leq k_0\leq \omega$,
let $1\leq\alpha_{\mfrak}\leq\alpha_{\pfrak}\leq\alpha_1\le\cdots
\le\alpha_9$
be a sequence of successor ordinals, and $\kappa_9<\kappa_8<\kappa_7$ compact cardinals with $\kappa_9>\alpha_9$.
Then there is a poset $P$ which forces (1) and (2) of Lemma~\ref{11values} for $\lambda_\mfrak=\aleph_{\alpha_\mfrak}$
and, in addition,
        \begin{multline*}
    \pfrak=\aleph_{\alpha_{\pfrak}},\ 
    \addN=\aleph_{\alpha_1},\ \covN=\aleph_{\alpha_2},\ \mathfrak{b}=\aleph_{\alpha_3},\ \nonM=\aleph_{\alpha_4},\\
        \covM=\aleph_{\alpha_5},\ \mathfrak{d}=\aleph_{\alpha_6},\ \nonN=\aleph_{\alpha_7},\ \cofN=\aleph_{\alpha_8},\text{\ and }\mathfrak{c}=\aleph_{\alpha_9}.
   \end{multline*}
\end{theorem}


\begin{figure}
  \centering
\[
\xymatrix@=3.5ex{
&&&&            \aleph_5\ar[rdd] & \aleph_7 \ar[rdd]      &  \mye\ar@{=}[d]\ar[ddr]      & \aleph_{11}\ar[r] &\aleph_{12} \\
&&&&                               & \aleph_6\ar[u]  &  \aleph_9 &              \\
\aleph_1\ar[r] & \aleph_2\ar[r] & \aleph_3\ar@/^5mm/[rr]&\mye
& \aleph_4\ar[uu] & \mye\ar@{=}[u] &  \aleph_8\ar[u]& \aleph_{10}\ar[uu]
}
\]
    \caption{\label{fig:increasing}A possible assignment for
     Figure~\ref{fig:result} (note that we lose control of $\hfrak$): $\mfrak=\aleph_2$, $\pfrak=\aleph_3$,
     $\lambda_i=\aleph_{3+i}$ for $i=1,\dots,9$.}
\end{figure}
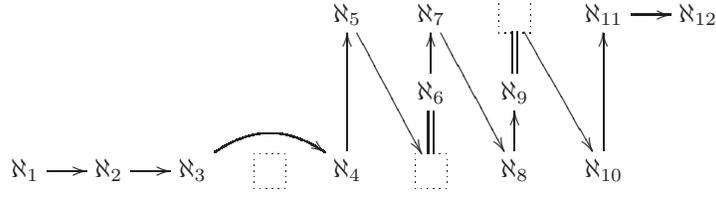

Actually, we will prove something more general:
We first formulate this more general result for the case $\alpha_1<\alpha_2<\alpha_3$; as explained in Remark~\ref{rem:and-or-equal}, there are variants of the theorem which allow $\alpha_1=\alpha_2$ and/or $\alpha_2=\alpha_3$.

\begin{theorem}\label{nogapsA}
   Assume GCH and $1\leq k_0\leq\omega$.
   Let $1\leq\alpha_{\mfrak}\leq\alpha_{\pfrak}\leq\alpha_1<\alpha_2<\alpha_3\leq\alpha_4\leq\ldots\leq\alpha_9$ be ordinals and assume that there are strongly compact cardinals $\kappa_9<\kappa_8<\kappa_7$ such that
   \begin{enumerate}[(i)]
       \item $\alpha_{\pfrak}\leq\kappa_9$, $\alpha_1<\kappa_8$ and $\alpha_2<\kappa_7$;
       \item for $i=1,2,3$, $\aleph_{\beta_{i-1}+(\alpha_i-\alpha_{i-1})}$ is regular,\footnote{This is equivalent to say that $\alpha_i$ is either a successor ordinal or a weakly inaccessible larger than $\beta_{i-1}$.} where $\beta_i:=\max\{\alpha_i,\kappa_{10-i}+1\}$ and $\alpha_0=\beta_0=0$;
       \item for $i\geq4$, $i\neq 6$, $\aleph_{\beta_3+(\alpha_i-\alpha_3)}$ is regular;
       \item $\cof(\aleph_{\beta_3+(\alpha_6-\alpha_3)})\geq\aleph_{\beta_3}$; and
       \item $\aleph_{\alpha_{\mfrak}}$ and $\aleph_{\alpha_{\pfrak}}$ are regular.
   \end{enumerate}
   Then we get a poset $P$ as in the previous Theorem.
\end{theorem}
\begin{proof}
    For $4\leq i\leq 9$ put $\beta_i:=\beta_3+(\alpha_i-\alpha_3)$. Also set $\lambda_{\mfrak}:=\aleph_{\alpha_{\mfrak}}$, $\lambda_{\pfrak}:=\aleph_{\alpha_{\pfrak}}$ and $\lambda_i:=\aleph_{\beta_i}$ for $1\leq i\leq 9$. Note that $\lambda_i$ is regular for $i\neq 6$, $\cof(\lambda_6)\geq\lambda_3$ and $\lambda_{\mfrak}\leq\lambda_{\pfrak}\leq\kappa_9<\lambda_1<\kappa_8<\lambda_2<\kappa_7<\lambda_3\leq\lambda_4\leq\ldots\leq\lambda_9$. In the case $\alpha_{\pfrak}<\alpha_1$ let $P$ be the $\lambda_{\pfrak}^+$-cc poset corresponding to Theorem~\ref{thm:bla66} (the modification of $P^{\vA*}$), otherwise let $P$ be the ccc poset corresponding to Lemma~\ref{11values} and forcing $\pfrak\geq\kappa_9$ and $\mfrak(k_0\text{-Knaster})=\mfrak(\textnormal{precaliber})=\lambda_{\mfrak}$ (or just $\mfrak(k_0\text{-Knaster})\geq\kappa_9$ when $\alpha_{\mfrak}=\kappa_9$).\footnote{This distinction is necessary: the forcing $P$ from Theorem~\ref{thm:bla66} is not $\lambda_\pfrak$-cc, so we would not be able to apply Lemma~\ref{lem:collapse} to $P\times\col(\lambda_\pfrak,\kappa_9^+)$.}

    \emph{Step 1.} We first assume $\alpha_{\pfrak}<\alpha_1$. In the case $\kappa_9<\alpha_1$ we have $\beta_1=\alpha_1$, so let $P_1:=P$; in the case $\alpha_1\leq\kappa_9$, we have $\beta_1=\kappa_9+1$ and $\lambda_1=\kappa_9^+$. Put $\kappa_1:=\aleph_{\alpha_1}$ and $P_1:=P\times\col(\kappa_1,\lambda_1)$. It is clear that $\kappa_1$ is regular and $\kappa_1\leq\lambda_1$ so, by Lemma~\ref{lem:collapse}, $P_1$ forces $\addN=\aleph_{\alpha_1}$ and that the values of the other cardinals are the same as in the $P$-extension. Even more, for any $\xi\geq\kappa_8$, $P_1$ forces $\aleph_\xi=\aleph^V_\xi$ because, in the ground model, $\kappa_8$ is an $\aleph$-fixed point between $\aleph_{\beta_1}$ and $\aleph_{\beta_2}$ (and thus between $\beta_1$ and $\beta_2$).

    Now assume $\alpha_{\pfrak}=\alpha_1$ (so $P$ is ccc) and let $\kappa_1:=\lambda_{\pfrak}=\aleph_{\alpha_1}$. Since $\alpha_1=\alpha_{\pfrak}\leq\kappa_9$, we have $\lambda_1=\kappa_9^+$, so we set $P_1:=P\times\col(\kappa_1,\lambda_1\times\kappa_1)$. This poset forces the same as the above, but for $\pfrak$ we just now $\pfrak\geq\kappa_1$ (or just $\mfrak(k_0\text{-Knaster})\geq\kappa_9$ when $\alpha_p=\kappa_9$), but $\pfrak\leq\kappa_1$ also holds because $\col(\kappa_1,\lambda_1\times\kappa_1)$ adds a $\kappa_1^{<\kappa_1}$-generic function (see the proof of Lemma~\ref{fact:wreqwr}).

    \emph{Step 2.} In the case $\kappa_8<\alpha_2$ put $P_2:=P_1$; otherwise, we have $\beta_2=\kappa_8+1$ and $\lambda_2=\kappa_8^+$. Set $\kappa_2:=\aleph_{\beta_1+(\alpha_2-\alpha_1)}$ and $P_2:=P_1\times\col(\kappa_2,\lambda_2)$. It is clear that $\kappa_2<\lambda_2$, so
    Lemma~\ref{lem:collapse} applies, i.e., $P_2$ forces $\covN=\kappa_2$ and that the values of the other cardinals are the same as in the $P_1$-extension. Also note that $P_1$ forces $\kappa_2=\aleph_{\alpha_2}$, and this value remains unaltered in the $P_2$-extension. Furthermore $P_2$ forces $\aleph_\xi=\aleph^V_\xi$ for any $\xi\geq\kappa_7$.

    \emph{Step 3.} In the case $\kappa_7<\alpha_3$ put $P_3:=P_2$; otherwise, set $\kappa_3:=\aleph_{\beta_2+(\alpha_3-\alpha_2)}$ and $P_3:=P_2\times\col(\kappa_3,\lambda_3)$. Note that $P_3$ forces $\bfrak=\kappa_3=\aleph_{\alpha_3}$ and that the other values are the same as forced by $P_2$. Hence, $P_3$ is as desired, e.g., $\nonM=\lambda_4=\aleph_{\beta_4}^V=\aleph_{\alpha_4}$.
\end{proof}

\begin{remark}\label{rem:and-or-equal}
   Theorem~\ref{nogapsA} also holds when $\alpha_1\leq\alpha_2\leq\alpha_3$, but depending on the equalities the hypothesis may change. For example, in the case $\alpha_1=\alpha_2<\alpha_3$, hypothesis (ii) is modified by: $\beta_1=\kappa_9+1$, $\beta_2=\kappa_8+1$, $\beta_3=\max\{\alpha_3,\kappa_7+1\}$ and both $\aleph_{\alpha_1}$ and $\kappa_3:=\aleph_{\beta_2+(\alpha_3-\alpha_2)}$ are regular. For the proof, the idea is first collapse $\lambda_2:=\aleph_{\beta_2}$ to $\kappa_1:=\aleph_{\alpha_1}$ (as in step 1 of the proof, considering similar cases for $\alpha_{\pfrak}$), and then (possibly) collapse $\lambda_3:=\aleph_{\beta_3}$ to $\kappa_3$ (as in step 3).
   This guarantees that the sequence of cardinals in the previous theorem is $\le$-consistent.
\end{remark}

A similar result (and remark about $\le$-consistency) applies to \ref{versionB}.

\begin{theorem}\label{nogapsB}
   Assume GCH and $1\leq k_0\leq\omega$.
   Let $1\leq\alpha_{\mfrak}\leq\alpha_{\pfrak}\leq\alpha_1<\alpha_2<\alpha_3\leq\alpha_4\leq\ldots\leq\alpha_9$ be ordinals and assume that there are strongly compact cardinals $\kappa_9<\kappa_8<\kappa_7<\kappa_6$ such that
   \begin{enumerate}[(i)]
       \item $\alpha_{\pfrak}\leq\kappa_9$, $\alpha_1<\kappa_8$, $\alpha_2<\kappa_7$, and $\alpha_3<\kappa_6$;
       \item for $i=1,2,3,4$, $\aleph_{\beta_{i-1}+(\alpha_i-\alpha_{i-1})}$ is regular, where $\beta_i:=\max\{\alpha_i,\kappa_{10-i}+1\}$ and $\alpha_0=\beta_0=0$;
       \item for $i\geq6$, $\aleph_{\beta_4+(\alpha_i-\alpha_4)}$ is regular;
       \item $\cof(\aleph_{\beta_4+(\alpha_5-\alpha_4)})\geq\aleph_{\beta_4}$;
       \item $\beta_3$ is \underline{not} the successor of a cardinal with countable cofinality; and
       \item $\aleph_{\alpha_{\mfrak}}$ and $\aleph_{\alpha_{\pfrak}}$ are regular.
   \end{enumerate}
   Then there is a poset that forces (1) and (2) of Lemma~\ref{11values} for $\lambda_\mfrak=\aleph_{\alpha_\mfrak}$ and
   \begin{multline*}
    \pfrak=\aleph_{\alpha_{\pfrak}},\ 
    \addN=\aleph_{\alpha_1},\ \bfrak=\aleph_{\alpha_2},\ \covN=\aleph_{\alpha_3},\ \nonM=\aleph_{\alpha_4},\\
        \covM=\aleph_{\alpha_5},\ \nonN=\aleph_{\alpha_6},\ \dfrak=\aleph_{\alpha_7},\ \cofN=\aleph_{\alpha_8},\text{\ and }\mathfrak{c}=\aleph_{\alpha_9}.
   \end{multline*}
\end{theorem}

\bibliography{morebib}

\providecommand{\bysame}{\leavevmode\hbox to3em{\hrulefill}\thinspace}
\providecommand{\MR}{\relax\ifhmode\unskip\space\fi MR }
\providecommand{\MRhref}[2]{%
  \href{http://www.ams.org/mathscinet-getitem?mr=#1}{#2}
}
\providecommand{\href}[2]{#2}
\begin{thebibliography}{GKMS20}

\bibitem[BCM18]{diegoetal}
J\"{o}rg Brendle, Miguel~A. Cardona, and Diego~A. Mej\'{i}a,
  \emph{Filter-linkedness and its effect on preservation of cardinal
  characteristics}, \href{https://arxiv.org/abs/1809.05004}{arXiv:1809.05004},
  2018.

\bibitem[Bla93]{MR1234278}
Andreas Blass, \emph{Simple cardinal characteristics of the continuum}, Set
  theory of the reals ({R}amat {G}an, 1991), Israel Math. Conf. Proc., vol.~6,
  Bar-Ilan Univ., Ramat Gan, 1993, pp.~63--90. \MR{1234278}

\bibitem[Bla10]{Blass}
\bysame, \emph{Combinatorial cardinal characteristics of the continuum},
  Handbook of set theory. {V}ols. 1, 2, 3, Springer, Dordrecht, 2010,
  pp.~395--489. \MR{2768685}

\bibitem[Bre91]{Br}
J{\"o}rg Brendle, \emph{Larger cardinals in {C}icho\'n's diagram}, J. Symbolic
  Logic \textbf{56} (1991), no.~3, 795--810. \MR{1129144 (92i:03055)}

\bibitem[DJ82]{MR667224}
A.~J. Dodd and R.~B. Jensen, \emph{The covering lemma for {$L[U]$}}, Ann. Math.
  Logic \textbf{22} (1982), no.~2, 127--135. \MR{667224}

\bibitem[DS]{longlow}
Alan Dow and Saharon Shelah, \emph{On the bounding, splitting, and
  distributivity numbers of {$\mathcal P (\mathbb N)$}; an application of
  long-low iterations}, \url{https://math2.uncc.edu/~adow/F1276.pdf}.

\bibitem[{Fre}84]{zbMATH03891346}
D.~H. {Fremlin}, \emph{{Cicho\'n's diagram.}}, {Publ. Math. Univ. Pierre Marie
  Curie 66, S\'emin. Initiation Anal. 23\`eme Ann\'ee-1983/84, Exp. No. 5, 13
  p. (1984).}, 1984.

\bibitem[GKMS]{GKMS1}
Martin Goldstern, Jakob Kellner, Diego~A. Mejia, and Saharon Shelah,
  \emph{Controlling cardinal characteristics without adding reals},
  \href{https://arxiv.org/abs/2006.09826}{arXiv:2006.09826}.

\bibitem[GKMS20]{GKMS2}
Martin Goldstern, Jakob Kellner, Diego~Alejandro Mej\'{i}a, and Saharon Shelah,
  \emph{Cicho\'n's maximum without large cardinals}, J. Eur. Math. Soc. (JEMS)
  \textbf{to appear} (2020),
  \href{https://arxiv.org/abs/1906.06608}{arXiv:1906.06608}.

\bibitem[GKS19]{GKS}
Martin Goldstern, Jakob Kellner, and Saharon Shelah, \emph{Cicho\'{n}'s
  maximum}, Ann. of Math. \textbf{190} (2019), no.~1, 113--143,
  \href{https://arxiv.org/abs/1708.03691}{arXiv:1708.03691}.

\bibitem[GMS16]{GMS}
Martin Goldstern, Diego~Alejandro Mej\'{i}a, and Saharon Shelah, \emph{The left
  side of {C}icho\'{n}'s diagram}, Proc. Amer. Math. Soc. \textbf{144} (2016),
  no.~9, 4025--4042. \MR{3513558}

\bibitem[Jec03]{Je03}
Thomas Jech, \emph{{Set theory}}, Springer Monographs in Mathematics,
  Springer-Verlag, Berlin, 2003, The third millennium edition, revised and
  expanded.

\bibitem[KST19]{KeShTa:1131}
Jakob Kellner, Saharon Shelah, and Anda T{\u{a}}nasie, \emph{Another ordering
  of the ten cardinal characteristics in cicho\'{n}'s diagram}, Comment. Math.
  Univ. Carolin. \textbf{60} (2019), no.~1, 61--95.

\bibitem[KTT18]{KTT}
Jakob Kellner, Anda~Ramona T\u{a}nasie, and Fabio~Elio Tonti, \emph{Compact
  cardinals and eight values in {C}icho\'{n}'s diagram}, J. Symb. Log.
  \textbf{83} (2018), no.~2, 790--803. \MR{3835089}

\bibitem[Mej19]{modKST}
Diego~A. Mej\'{i}a, \emph{A note on ``{Another ordering of the ten cardinal
  characteristics in Cicho\'n's Diagram}'' and further remarks}, Ky\={o}to
  Daigaku S\=urikaiseki Kenky\=usho K\=oky\=uroku \textbf{2141} (2019), 1--15,
  \href{https://arxiv.org/abs/1904.00165}{arXiv:1904.00165}.

\bibitem[Mil98]{MR1625907}
Heike Mildenberger, \emph{Changing cardinal invariants of the reals without
  changing cardinals or the reals}, J. Symbolic Logic \textbf{63} (1998),
  no.~2, 593--599. \MR{1625907}

\bibitem[MS16]{MSpt}
M.~Malliaris and S.~Shelah, \emph{Cofinality spectrum theorems in model theory,
  set theory, and general topology}, J. Amer. Math. Soc. \textbf{29} (2016),
  no.~1, 237--297. \MR{3402699}

\bibitem[Voj93]{MR1234291}
Peter Vojt\'{a}\v{s}, \emph{Generalized {G}alois-{T}ukey-connections between
  explicit relations on classical objects of real analysis}, Set theory of the
  reals ({R}amat {G}an, 1991), Israel Math. Conf. Proc., vol.~6, Bar-Ilan
  Univ., Ramat Gan, 1993, pp.~619--643. \MR{1234291}

\end{thebibliography}
\bibliographystyle{amsalpha}
\end{document}